\documentclass[12pt, a4paper]{amsart}
\usepackage{amsfonts,amsmath, amsthm, amssymb, amscd}
\usepackage{latexsym}
\input{amssym.def}
\input{amssym}

\usepackage[dvips]{graphicx}
\usepackage{psfrag}
\usepackage{comment}

\newtheorem{theorem}{Theorem}
\newtheorem{lemma}[theorem]{Lemma}
\newtheorem{corollary}[theorem]{Corollary}

\newtheorem{definition}[theorem]{{Definition}}

\numberwithin{equation}{section}

\newcommand {\N}{\mathbb{N}} 
\newcommand {\R}{\mathbb{R}} 

\def\remark{\par\medbreak\noindent{\sc Remark}\quad\enspace}

\begin{document}

\title[Anosov flows on non-compact Riemannian manifolds ]{ A Note on Anosov flows of non-compact Riemannian manifolds }

\author{Gerhard Knieper }
\date{\today}
\address{Faculty of Mathematics,
Ruhr University Bochum, 44780 Bochum, Germany}
\email{gerhard.knieper@rub.de}
\thanks {This work was partially supported by the German Research Foundation (DFG),  CRC TRR 191,  {\it Symplectic
structures in geometry, algebra and dynamics.} }
\begin{abstract}
In this note we formulate a condition for complete non-compact Riemannian manifolds which implies no conjugate points
in case that the geodesic flow is Anosov with respect to the Sasaki metric.
\end{abstract}

\maketitle
In 1974 Klingenberg \cite{Kl1} proved, using Morse theory, that geodesic Anosov flows on compact Riemannian manifolds do not have
conjugate points.   In 1987  Man\'{e} \cite{Ma} showed with the help of the Maslov index that geodesic flows on complete Riemannian manifolds with finite volume
have no conjugate points provided there exists a continuous flow invariant Lagrangian section. Since the stable and unstable
bundles provide such invariant sections, geodesic Anosov flows on manifolds with finite volume do not have conjugate points as well.
In the same paper, Man\'{e} claimed that  any complete non-compact Riemannian manifold with lower curvature bound has no conjugate points
in case the geodesic flow is Anosov with respect to the Sasaki metric. However, as we noticed in \cite{Kn1}, the proof contains a mistake
which occurs in proposition II.2 of his article.\\

This note was inspired by a question of the authors of \cite{GUS}, whether one could prove Man\'{e}'s claim under certain extra geometric conditions.
The positive answer to this question would allow them to remove an assumption in their work. To formulate this result, we start by introducing some notations and definitions.

In the following, $(M,g)$ will denote a complete Riemannian manifold, $ \pi:TM \to M$ the tangent bundle with the canonical projection and
$
SM = \{v \in TM \;|\; \|v\| = 1\},
$
 the unit tangent bundle with respect to the Riemannian metric $g$.
The tangent space $T_vSM$ of $SM$ at $v \in SM$ is given by
$$
\{(x,y) \;|\; x,y \in T_{\pi(v)} M , y \perp v\},
$$
where we use the splitting of $T_vTM$ into horizontal and vertical spaces.
Using this decomposition,
the linearization of the geodesic
$$
D \phi^t(v) : T_{\pi v}M \times T_{\pi v} M \to T_{\pi \phi^t(v)}
 M \times T_{\pi \phi^t (v)} M
$$ 
is given by $D\phi^t (v)(x,y) = (J(t), J'(t)) $, where $J(t)$ is the Jacobi
field which is the 
solution of the Jacobi equation  
$$
J''(t) + R(J(t), \phi^t(v)) \phi^t(v) = 0,
$$
along
 $c_v$ 
with $J(0) = x, J'(0) = y$. Here $J' = \frac{D}{dt} J$ denotes the covariant derivative along $c_v$ and $R(X,Y) Z$ the Riemannian curvature tensor.
There is a natural 1-form $\Theta$ on $TM$ defined by 
$$
 \Theta_v(\xi) = \langle v,d\pi_v(\xi)\rangle.
$$
Using the
decomposition of $\xi = (x,y)$ into horizontal and vertical part,
introduced above, we obtain  
$$
\Theta_v(x,y) = \langle v,x\rangle.
$$
The differential $d\Theta$ is the canonical symplectic form and is given by 
$$
d\Theta_v ((x_1, y_1), (x_2, y_2)) = \langle y_1, x_2\rangle - \langle y_2, x_1\rangle.
$$
The
 canonical metric $ g_S$ on $TM$ given by
$$
g_S((x_1, y_1), (x_2, y_2)) = \langle x_1, x_2 \rangle +  \langle y_1, y_2 \rangle 
$$ 
for $(x_1, y_1), (x_2, y_2) \in T_vTM$  is called the Sasaki metric on $TM$.
Denote by
$$
E^{\phi}(v) = \{(\lambda v , 0) \;| \; \lambda  \in \mathbb{R} \}
$$
the 1-dimensional space tangent to the geodesic flow at $v \in SM$. The orthogonal complement 
$$
N(v) := E^{\phi}(v)^\perp
$$
in $T_vSM$ with respect to the Sasaki metric  defines a bundle $N$ invariant 
under the linearization
of the geodesic flow. Furthermore, $N$ is a symplectic bundle, i.e., 
the symplectic form restricted to $N$
is non-degenerate. A Lagrangian subspace $L(v) \subset N(v)$ is called a Lagrangian graph if
$V(v) \cap L(v) = \emptyset$.
\begin{lemma} \label{A}
If $L( \phi^t(v)) =D \phi^t(v) L(v)$ is a Lagrangian graph for all $ t \in [a,b]$, then $c_v:[a,b] \to M$ has no conjugate points.
\end{lemma}
\begin{proof}
 For a proof see Lemma 2.7 in \cite{Kn1}
 \end{proof}

\begin{definition}\label{B}
Let $(M,g)$ be a complete Riemannian manifold and 
$\| \cdot \|$ the norm
 on $TSM$ induced by the Sasaki metric. The geodesic flow 
$\phi^t: SM \to SM$ is called Anosov flow if there exist constants $k, C > 0$ 
and  a splitting
$$
T_vSM = E^s(v) \oplus E^u(v) \oplus E^\phi(v)
$$
such that $E^\phi(v) = \mathrm{span} \{X_G (v)\}$ and 
$$
\|D \phi^t(v) \xi \| \le C \cdot e^{-kt} \| \xi \|
$$
for all $\xi \in E^s(v), t \ge 0$, as well as
$$
\| D\phi^{-t} (v) \xi \| \le C e^{-kt} \|\xi\|
$$
for all $\xi \in E^u(v)$, $t \ge 0$. 
\end{definition}

\begin{lemma}\label{C}
Let $(M, g)$ be a complete Riemannian manifold with lower sectional curvature bound $-\beta^2$. If the
geodesic flow is Anosov with constants $k, C > 0$ as in Definition \ref{B}, there exists a constant $\sigma =\sigma(\beta, k, C) $ with the following property. If 
$$
E^s(v) \cap V(v) \not= \{0\}
$$ 
then the geodesic $c_v$ has
conjugate points on the interval $[-1, \sigma]$.
\end{lemma}
\begin{proof}
 For a proof see Lemma 2.14 in \cite{Kn1}
 \end{proof}
 \remark The Lemma above is related to proposition II.2 in  
Man\'{e}'s article \cite{Ma}. In this proposition he even claimed that $c_v(0)$ should be conjugated to $c_v(t_0)$ for some $t_0 >0$. 
However, the proof which is based on the Index form needs information about the geodesic on an intervall containing $0$ as an interior point.

 The following Theorem contains the main result of this paper.
\begin{theorem}
Let  $(M,g)$ be a complete, connected and non-compact Riemannian manifold with sectional curvature  bounded from below by  $- \beta^2$
 such that the geodesic flow is Anosov with respect to the Sasaki metric with constants $k,C$ as in Definition \ref{B}.
 Assume that the following three conditions are satisfied
\begin{enumerate}
\item
For all $v \in SM$ there exists an open neighborhood
$U(v) \subset SM$ such that
$
\lim_{t \to \infty}d(c(0), c(t))  = \infty
$
uniformly  for all geodesic $c:\R \to M$  with $\dot c(0) \in U(v) $.
\item
There  exists a compact set  $K \subset M$ such that for all $p \in M\setminus K$ and all geodesics $c$ with $c(0) =p$ the segment $c:[-1, \sigma] \to M$ has no conjugate points
with $\sigma = \sigma(\beta, k, C) $  as in Lemma \ref{C}.
\item There exists at least one geodesic without conjugate points.
\end{enumerate}
Then $(M,g)$  has no conjugate points.
\end{theorem}
   \remark Note, that the existence of a geodesic without conjugate points is guaranteed 
   if their exists a  geodesic $c_v$ which does not intersect  $K$. Otherwise,
 by Lemma  \ref{A} there would exist $t \in \R$ such that  $E^s( \phi^{t}(v))\cap V(\phi^{t}(v) ) \not= \emptyset$ and from Lemma \ref{C}  
 follows that $c_{v}:[-1 + t, \sigma +t] \to M$ has conjugate points
  contradicting the assumption that  $c_{v}$ is in the complement of $K$. 
 \begin{proof}
Consider the set 
$$
C(SM) = \{v \in SM \mid c_v: \R \to M \; \text{has no conjugate points} \}. 
$$
  It is known that $C(SM)$ is closed. For a proof see \cite{Ma}.
  Now we show that $C(SM)$ is open. To prove this, consider a sequence $v_n \in SM\setminus C(SM)$ converging to $v$.
  Then by Lemma  \ref{A} there is 
  $t_ n  \in \R$ such that $E^s( \phi^{t_n}(v_n))\cap V(\phi^{t_n}(v_n) ) \not= \emptyset$. 
  Lemma \ref{C} implies that $c_{v_n}:[-1 + t_n, \sigma +t_n] \to M$
  has conjugate points and from condition (2) we conclude  $c_{v_n}(t_n) \in K$. By condition (1)
  there exist  open neighborhoods $U(v), U(-v) \subset SM$ of $v$ such that
 $$
 \lim_{t \to \infty} d(c_w(t), c_w(0)) = \infty
  $$
  for all $w \in U(v)$ and   $w \in U(-v)$.
  This implies the existence of some $T >0$ such that for all $n \in \N$ with $v_n \in U(v)$ and $-v_n \in U(-v)$ we
  have  $|t_n| \le T$ and, therefore, $c_v: [-1+T,  \sigma +T) \to M$ has conjugate points as well.
  This shows $C(SM)$ is open and since by condition (3)  the set $C(SM)$  is non-empty we have $ C(SM) =SM$. Hence $(M,g)$ has no conjugate points.
 \end{proof}
 
 \begin{corollary}
 Let  $(M,g)$ be a complete, connected  and non-compact Riemannian manifold such that the sectional curvature is bounded from below by  $- \beta^2$
and such that the geodesic flow is Anosov with respect to the Sasaki metric with constants $k,C$ as in Definition \ref{B}.
 Assume that the following three conditions are satisfied
\begin{enumerate}
\item
For all $v \in SM$ there exists an open neighborhood
$U(v) \subset SM$ such that
$
\lim_{t \to \infty}d(c(0), c(t))  = \infty
$
uniformly  for all geodesic $c:\R \to M$  with $\dot c(0) \in U(v) $.
\item
There exists a ball $B(p,r)$ of radius $r$ about $p$ such that the sectional curvature  on
$M \setminus B(p,r)$ is smaller than
$(\frac{\pi}{\sigma +1})^2$ with $\sigma = \sigma(k,C, \beta) $ as in Lemma \ref{C}.
\item
There exists a geodesic without conjugate points.
\end{enumerate}
Then $(M,g)$  has no conjugate points.

 \end{corollary}
  
  \begin{proof}
  Consider  $r>0$  and $B(p,r)$ such that the sectional curvature  on $M \setminus B(p,r)$ is smaller than $(\frac{\pi}{\sigma +1})^2$. Then each geodesic  $c:[0,\sigma +1 ] \to M \setminus B(p,r)$
  has no conjugate points and the set $K =B(p,r+1)$ fulfills the assumption in the theorem above. In particular, $(M,g)$ has no conjugate points.
  \end{proof}
   \remark With the same argument as in the remark above a geodesic without conjugate points exists,  provided there is a geodesic which does not intersect $B(p,r)$.

 \end{document}